\documentclass{amsart}
\usepackage{amssymb}
\usepackage{amsmath}
\usepackage{tikz-cd}
\usepackage{amsthm}
\theoremstyle{plain}

\theoremstyle{plain}
\newtheorem{thm}{Theorem}
\newtheorem{prop}{Proposition}
\newtheorem{lem}{Lemma}
\newtheorem{coro}{Corollary}

\usepackage{geometry}
\theoremstyle{definition}

\newtheorem{rem}{Remark}

\newtheorem{definition}{Definition}

\begin{document}
\setcounter{page}{1}

\title[topological lattice rings with $AM$-property]{ topological lattice rings with $AM$-property}

\author[Omid Zabeti ]{Omid Zabeti}

\address{ Department of Mathematics, University of Sistan and Baluchestan, P.O. Box: 98135-674, Zahedan, Iran.}
\email{{o.zabeti@gmail.com}}



\date{Received: xxxxxx; Revised: yyyyyy; Accepted: zzzzzz.}

\begin{abstract}
Motivated by the recent definition of $AM$-property in locally solid vector lattices [O. Zabeti, arXiv: 1912.00141v2 [math.FA]], in this note, we try to investigate those results in the category of all locally solid lattice rings.
In fact, we characterize locally solid lattice rings in which order bounded sets and bounded sets agree. Furthermore, with the aid of $AM$-property, we find conditions under that, order bounded group homomorphisms and different types of bounded group homomorphisms coincide. Moreover, we show that each class of bounded order bounded group homomorphisms on a locally solid lattice ring $X$ has the Lebegsue or the Levi property if and only if so is $X$.
\end{abstract}
\date{\today}
\maketitle

\section{motivation and Preliminaries}
Let us start with some motivation. In general, combination between different aspects of mathematics usually arise more efficient results and applications. For example, a topological vector space is a combination between linear algebra and topology. Moreover, a locally solid vector lattice is a powerful connection between ordered sets, linear algebra, and topology. These notions have been studied sufficiently because many classical examples in functional analysis fit in this category. Among these objects, there are other topological algebraic structures that possess order connection, too. This leads us to the theory of ordered groups and ordered rings. When we add appropriate topological connections to them, we obtain more fruitful structures, for example locally solid lattice groups and locally solid lattice rings. These concepts are almost unexplored with respect to locally solid vector lattices although there are many applicable examples of them that fail to have either a vector space or a topological vector space structure for example discrete topology, box topology on product spaces, the multiplicative group $S^1$, the integers, and so on.

So, it is of independent interest to discover these phenomena. Recently, a suitable reference regarding lattice ordered groups has been announced in \cite{H}. Furthermore, lattice ordered rings is partially considered in \cite{J}.

On the other hand, in \cite{MZ}, it is shown that there are several types of bounded group homomorphisms between topological rings; with respect to suitable topologies, each class of them forms a topological ring, too. Moreover, when $X$ is a locally solid lattice ring, each class of bounded order bounded group homomorphisms, also, forms a locally solid lattice ring. This is done in \cite{Z2}, recently. Before, we proceed with some preliminaries, let us again present some detailed motivation.
It is worthwhile to mention that although it might seem at the first glance that there is no advantage in topological groups and topological rings with respect to topological vector spaces, but there are some less considered facts about them. For example, we know that the discrete topology is the most powerful topology but the only topological vector space with the discrete topology is the zero one. On the other hand, any group with the discrete topology forms a topological group.
Furthermore, box topology is important in product spaces because of finer neighborhoods with respect to the product topology and also to construct counterexamples; nevertheless, product of topological vector spaces with the box topology is not a topological vector space but this happens for product of topological groups.

The known Hahn-Banach theorem that relies on the scalar multiplication, appears in many situations when we are dealing with locally convex spaces. The bad news is that we lack it in the category of topological groups and there is no fruitful tool we can use it instead. Furthermore, many results regarding $AM$-property and applications utilize this theorem in their nature. So, we can not expect those results in the setting of topological groups, directly. The good news is that when we are working with topological rings, the multiplication is a handy tool in this way which turns out to be the right object for our purpose.
In fact, the main aim of this note, is to characterize rings and also group homomorphisms in which bounded and order bounded notions agree. This is done by using the concept "$AM$-property" that is defined at first in \cite{Z4} in the category of all locally solid vector lattices. Moreover, as an application, we show that each class of bounded order bounded group homomorphisms defined on a locally solid lattice ring $X$, has the Lebesgue or the Levi property if and only if so is $X$. The lattice structures for these classes of homomorphisms have been obtained recently in \cite{Z2}.

Suppose $X$ is a topological group. A set $B\subseteq G$ is said to be {\bf bounded} if for each neighborhood $U$ at the identity, there is a positive integer $n$ with $B\subseteq nU$ in which $nU=\{x_1+\ldots+x_n: x_i\in U\}$.

A  lattice group ( $\ell$-group)  $G$ is called {\bf order complete} if every non-empty bounded above subset of $G$ has a supremum. $G$ is {\bf Archimedean} if $nx\leq y$ for each $n\in \Bbb N$ implies that $x\leq 0$. It is easy to see that every order complete $\ell$-group is Archimedean. A set $S\subseteq G$ is called {\bf solid} if $x\in G$, $y\in S$ and $|x|\leq |y|$ imply that $x\in S$. Also, recall that a group topology $\tau$ on an $\ell$-group $G$ is referred to as  {\bf locally solid} if it has a local basis at the identity consisting of solid sets.


Suppose $G$ is a locally solid $\ell$-group. A net $(x_{\alpha})\subseteq G$ is said to be {\bf order} convergent to $x\in G$ if there exists a net $(z_{\beta})$ ( possibly over a different index set) such that $z_{\beta}\downarrow 0$ and for every $\beta$, there is an $\alpha_0$ with $|x_{\alpha}-x|\leq z_{\beta}$ for each $\alpha\ge \alpha_0$. A set $A\subseteq G$ is called {\bf order closed} if it contains limits of all order convergent nets which lie in $A$.
 Keep in mind that topology $\tau$ on a locally solid $\ell$-group $(G,\tau)$ is referred to as {\bf Fatou} if it has a local basis at the identity consists of solid order closed neighborhoods.
 Observe that a locally solid $\ell$-group $(G,\tau)$ is said to have the {\bf Levi property} if every $\tau$-bounded upward directed set in $G_{+}$ has a supremum.
 Finally, recall that a locally solid $\ell$-group $(G,\tau)$ possesses the {\bf Lebesgue property} if for every net $(u_{\alpha})$ in $G$, $u_{\alpha}\downarrow 0$ implies that $u_{\alpha}\xrightarrow{\tau}0$.
For undefined expressions and related topics, see \cite{H,Z3}.

Now, suppose $X$ is a topological ring. A set $B\subseteq X$ is called {\bf bounded} if for each zero neighborhood $W$, there is a zero neighborhood $V$ with $VB\subseteq W$ and $BV\subseteq W$.
 A lattice ring ( $\ell$-ring) is a ring that is also a lattice where the ring multiplication and the lattice structure are compatible via the inequality $|x.y|\leq |x|.|y|$. By a topological $\ell$-ring, we mean a topological ring which is an $\ell$-ring, simultaneously. Moreover, observe that a locally solid $\ell$-ring is a topological $\ell$-ring that possesses a local basis consisting of solid sets. Also, note that since in this case, the underlying topological group is also locally solid, all of the properties regarding locally solid $\ell$-groups, mentioned above, can be transformed directly to the category of all locally solid $\ell$-rings; because in this case, order structure in a ring and the underlying group is the same, just, we need to replace boundedness in some statements with the one related to topological rings. Moreover, note that by an {\bf ideal} $I$ of an $\ell$-ring $X$, we mean a solid subring of $X$.

 Suppose $X$ is a locally solid $\ell$-ring. Then, it is called a  Birkhoff and Pierce ring ($f$-ring) if it satisfies in this property: $a\wedge b = 0$
and $c > 0$ imply that $ca\wedge b = ac\wedge b = 0$. For ample facts regarding this subject, see \cite{J}.

For a brief but informed context related to topological lattice rings, we refer the reader to \cite{Z2}.





\section{main results}

{\bf Observation}. Suppose $G$ is an Archimedean $\ell$-group. For every subset $A$, by $A^{\vee}$, we mean the set of all finite suprema of elements of $A$; more precisely,
$A^{\vee}=\{a_1\vee\ldots\vee a_n: n\in \Bbb N, a_i\in A\}$. It is obvious that $A$ is bounded above in $G$ if and only if so is $A^{\vee}$ and in this case, when the supremum exists, $\sup A=\sup A^{\vee}$. Moreover, put $A^{\wedge}=\{a_1\wedge\ldots\wedge a_n: n\in \Bbb N, a_i\in A\}$. It is easy to see that $A$ is bounded below if and only if so is $A^{\wedge}$ and $\inf A=\inf A^{\wedge}$ ( when the infimum exists). Observe that  $A^{\vee}$ can be viewed as an upward directed set in $G$ and $A^{\wedge}$ can be considered as a downward directed set.

Suppose $G$ is a locally solid $\ell$-group. We say that $G$ has {\bf $AM$-property} provided that for every bounded set $B\subseteq G$, $B^{\vee}$ is also bounded. It is worthwhile to mention that when $B$ is bounded and solid, $B^{\vee}$ is bounded if and only if $B^{\wedge}$ is bounded; this follows from the fact that $G$ is locally solid and $x_1\wedge\ldots\wedge x_n=-((-x_1)\vee\ldots\vee(-x_n))$ for any $n\in \Bbb N$ and any $x_i\in B$. One can consider this definition exactly for Archimedean $\ell$-rings. Note that when the ring multiplication is zero, every locally solid $\ell$-ring possesses $AM$-property.
This definition was originally defined in \cite{Z4} for locally solid vector lattices.

Let us first prove a version of \cite[Theorem 3.1]{Z} for topological rings.
\begin{thm}\label{17}
Let $(X_{\alpha})_{\alpha \in A}$ be a family of topological rings and $X=\prod_{\alpha\in A}X_{\alpha}$ with the product topology and coordinate-wise multiplication. Then $B\subseteq X$ is bounded if and only if there exists a family of subsets $(B_{\alpha})_{\alpha\in A}$ such that each $B_{\alpha}\subseteq X_{\alpha}$ is bounded and $B\subseteq \prod_{\alpha\in A}B_{\alpha}$.
\end{thm}
\begin{proof}
Suppose $B\subseteq X$ is bounded. Put
\begin{center}
$B_{\alpha}=\{x\in X_{\alpha}: \exists y=(y_{\beta})\in B   $ {and} $   x  $ {is} $ {\alpha} $ {-th coordinate of} $ y\}.$
\end{center}
Each $B_{\alpha}$ is bounded. For, if $U_{\alpha}$ is a zero neighborhood in $X_{\alpha}$, put
\[U=U_{\alpha}\times\prod_{\beta\neq \alpha}X_{\beta}.\]
Indeed, $U$ is a zero neighborhood in $X$. Therefore, there is a zero neighborhood $V$ with $VB\subseteq U$. Suppose $V_{\alpha}$ is the $\alpha-th$ component of $V$; it is clear that $V_{\alpha}B_{\alpha}\subseteq U_{\alpha}$.

For the converse, assume that there is a net $(B_{\alpha})_{\alpha\in A}$ of bounded sets with $B_{\alpha}\subseteq X_{\alpha}$ such that $B\subseteq \prod_{\alpha\in A}B_{\alpha}$. It is enough to show that $\prod_{\alpha\in A}B_{\alpha}$ is bounded. Assume that $U$ is an arbitrary zero neighborhood in $X$. So, $U=\prod_{\alpha\in A}U_{\alpha}$ in which $U_{\alpha}=X_{\alpha}$ for all but finitely many $\alpha$; namely, $U_{\alpha_i}\neq X_{\alpha_i}$ for $i\in \{1,2,\ldots,n\}$. Find zero neighborhoods $V_{\alpha_i}$ with $V_{\alpha_i}B_{\alpha_i}\subseteq U_{\alpha_i}$. Put $V=\prod_{i=1}^{n}V_{\alpha_i}\times \prod_{\beta\neq \{\alpha_1,\ldots,\alpha_n\}}X_{\beta}$. It is now easy to see that $V(\prod_{\alpha\in A}B_{\alpha})\subseteq U$, as claimed.
\end{proof}
\begin{prop}\label{000}
Suppose $(X_{\alpha})_{\alpha\in A}$ is a family of locally solid $\ell$-rings. Put $X=\prod_{\alpha\in A}{X_{\alpha}}$ with the product topology, pointwise ordering, and coordinate-wise multiplication. If each $X_{\alpha}$ has $AM$ property, then so is $X$.
\end{prop}
\begin{proof}
Suppose $B\subseteq X$ is bounded. By Theorem \ref{17}, there exists a net $(B_{\alpha})_{\alpha\in A}$ such that for each $\alpha$, $B_{\alpha}\subseteq X_{\alpha}$ is bounded and $B\subseteq \prod_{\alpha\in A}B_{\alpha}$. We show that $B^{\vee}$ is also bounded. Let $W$ be an arbitrary zero neighborhood in $X$. So, there are zero neighborhoods $(U_{\alpha_i})_{i\in\{1,\ldots,n\}}$  such that $W=\prod_{i=1}^{n}U_{\alpha_i}\times \prod_{\beta\in A-\{\alpha_1,\ldots,\alpha_n\}}X_{\beta}$.

Observe that for each $x\in B$, there is a net $(x_{\beta})_{\beta \in A}$ with $x_{\beta}\in B_{\beta}$. Now, consider the set $\{x_1,\ldots,x_m\}\subseteq B$ in which $m\in \Bbb N$ is fixed but arbitrary. It is enough to show that $x_1\vee\ldots\vee x_m$ is also bounded. Note that
\[x_1\vee\ldots\vee x_m=(x_{\beta}^{1})\vee\ldots\vee(x_{\beta}^{m})=(x_{\beta}^{1}\vee\ldots\vee x_{\beta}^{m})_{\beta\in A}.\]
Where $x_{\beta}^{j}\in B_{\beta}$ for each $j\in\{1,\ldots m\}$. For each $i\in\{1,\ldots n\}$, $B_{\alpha_i}$ has $AM$-property so that choose zero neighborhoods $(V_{\alpha_i})_{i=1}^{n}$ such that $V_{\alpha_i}(x_{\alpha_i}^{1}\vee\ldots\vee x_{\alpha_i}^{m})\subseteq U_{\alpha_i}$. Find zero neighborhood $V$ with $V\subseteq \cap_{i=1}^{n}V_{\alpha_i}$. Then, it can be easily seen  that $V(x_1\vee\ldots\vee x_m)\subseteq W$, as claimed.
\end{proof}
\begin{prop}\label{111}
Suppose $(X_{\alpha})_{\alpha\in A}$ is a family of locally solid $\ell$-rings. Put $X=\prod_{\alpha\in A}{X_{\alpha}}$ with the product topology, pointwise ordering, and coordinate-wise multiplication. If each $X_{\alpha}$ has the Levi property, then so is $X$.
\end{prop}
\begin{proof}

Suppose $(x^{\beta})_{\beta \in B}$ is a bounded increasing net in $X$. We need to show that its supremum exists. Observe that for each $\beta$, $x^{\beta}=(x^{\beta}_{\alpha})_{\alpha\in A}$. Since $X$ has the product topology, we conclude that the net is pointwise bounded; more precisely, for each fixed $\alpha$, the net $(x^{\beta}_{\alpha})_{\beta\in B}$ is bounded and also increasing in $X_{\alpha}$ so that it has a supremum by the assumption, namely, $y_{\alpha}=\sup\{(x^{\beta}_{\alpha})_{\beta\in B}\}$. Now, it can be easily seen that $y=(y_{\alpha})_{\alpha\in A}=\sup \{(x^{\beta}_{\alpha})_{\alpha\in A,\beta\in B}\}$.

\end{proof}
Observe that Proposition \ref{111}, can be restated exactly for locally solid $\ell$-groups, too. Moreover, when we consider the box topology, we have the following observations. Just, recall that the product of any family of topological groups with respect to the box topology is again a topological group ( see \cite[Chapter 3, Exercise 9]{Taq}).
\begin{prop}
Suppose $(G_{\alpha})_{\alpha\in A}$ is a family of locally solid $\ell$-groups whose singletons are bounded. Put $G=\prod_{\alpha\in A}{G_{\alpha}}$ with the box topology and pointwise ordering. If each $G_{\alpha}$ has the $AM$ property, then so is $G$.
\end{prop}
\begin{proof}
Suppose $B\subseteq X$ is bounded. By \cite[Theorem 3.4]{Z}, there exists a family $(\alpha_i)_{i=1,\ldots,n}$ of indices such that $B\subseteq (\prod_{i=1}^{n}B_{\alpha_i})\times \prod_{\beta \in A-\{\alpha_1,\ldots,\alpha_n\}}\{e_{\beta}\}$. Consider a set $\{x_1,\ldots,x_m\}$ in $B$. For each $j=1,\ldots,m$, we can write $x_j=({x_{\beta}}^{j})_{\beta\in A}$ where for $\beta\in A-\{\alpha_1,\ldots,\alpha_n\}$, ${x_{\beta}}^{j}=e_{\beta}$ and ${x^{j}}_{\alpha_i}\in B_{\alpha_i}$ for $i=1,\ldots,n$. Therefore, $x_1\vee\ldots\vee x_m=({x_{\beta}}^{1}\vee \ldots\vee {x_{\beta}}^{m})_{\beta\in A}$. Thus, this supremum is the net consisting of $({x_{\alpha_1}}^{j}\vee\ldots\vee{x_{\alpha_n}}^{j})$ in the $j-th$-place for $j=1,\ldots,m$ and for other terms, the identity. By the assumption, we conclude that $B^{\vee}$ is also bounded.
\end{proof}
Furthermore, by considering this point that when a set in a product space is bounded in the box topology, it is bounded in the product topology and compatible with Proposition \ref{111}, we have the following.
\begin{coro}
Suppose $(G_{\alpha})_{\alpha\in A}$ is a family of locally solid $\ell$-groups. Put $G=\prod_{\alpha\in A}{G_{\alpha}}$ with the box topology and pointwise ordering. If each $G_{\alpha}$ has the Levi property, then so is $G$.
\end{coro}

Now, we recall some notes about bounded group homomorphisms between topological rings; for a detailed exposition on this concept, see \cite{MZ,Z2}.
\begin{definition}\rm
Let $X$ and $Y$ be topological rings. A group homomorphism $T:X \to
Y$ is said to be
\begin{itemize}
\item[$(1)$] \emph{{\sf nr}-bounded} if there exists a
zero neighborhood $U\subseteq X$  such that $T(U)$ is bounded in $Y$.

\item[$(2)$] \emph{{\sf br}-bounded} if for every bounded set $B
\subseteq X$, $T(B)$ is bounded in $Y$.
\end{itemize}
\end{definition}

The set of all {\sf nr}-bounded ({\sf br}-bounded) homomorphisms
from a topological ring $X$ to a topological ring $Y$ is denoted
by ${\sf Hom_{nr}}(X,Y)$ (${\sf Hom_{br}}(X,Y)$). The set of all continuous homomorphisms from $X$ into $Y$ will be denoted by  ${\sf Hom_{cr}}(X,Y)$.

${\sf Hom_{nr}}(X,Y)$ is equipped with the topology of
uniform convergence on some zero neighborhood; note that a net $(S_{\alpha})$ of ${\sf
nr}$-bounded homomorphisms converges uniformly on a neighborhood $U\subseteq X$ to a homomorphism $S$  if for each zero neighborhood $V\subseteq Y$  there exists an $\alpha_0$ such that for each
$\alpha\geq\alpha_0$, $(S_{\alpha}-S)(U)\subseteq V$. ${\sf Hom_{br}}(X,Y)$ is allocated to the topology of uniform convergence on bounded sets;
observe that a net $(S_{\alpha})$ of ${\sf
br}$-bounded homomorphisms uniformly converges to a homomorphism $S$
on a bounded set $B\subseteq X$ if for each zero neighborhood $V\subseteq Y$
there is an $\alpha_0$ with $(S_{\alpha}-S)(B) \subseteq V$ for
each $\alpha\ge \alpha_0$. ${\sf Hom_{cr}}(X,Y)$ is assigned with the
topology of ${\sf cr}$-convergence; a net $(S_{\alpha})$ of continuous homomorphisms ${\sf cr}$-converges
to a homomorphism $S$ if for each zero neighborhood $W\subseteq Y$, there
is a neighborhood $U\subseteq X$ such that for every zero neighborhood $V\subseteq Y$ there exists an $\alpha_0$ with
$(S_{\alpha}-S)(U)\subseteq VW$ for each $\alpha\geq\alpha_0$.

Each class of bounded homomorphisms as well as continuous homomorphisms between topological rings can have a topological ring structure ( see \cite{MZ} for more information). Moreover, bounded order bounded homomorphisms between topological lattice rings can have lattice structures, using a kind of the Riesz-Kantorovich formulae, this is investigated in \cite{Z2}.
\begin{rem}
It is known that every zero neighborhood in a topological vector space is absorbing so that singletons are bounded. This useful fact relies on the scalar multiplication that we lack in topological groups, certainly. Therefore, we can not expect in a topological group that singletons are bounded, in general. For example, consider the additive group $\Bbb R$ with the usual topology and the additive group $\Bbb Z$ with the discrete topology. Put $G=\Bbb R\times \Bbb Z$. It is easy to see that $(0,1)$ is not bounded in $G$. But in many classical groups, singletons are bounded; for example when $G$ is a connected topological group ( see \cite[Chapter 3, Theorem 6]{Taq}. Moreover, suppose $G$ is a locally convex topological vector space. So, we have two notions for boundedness in $G$; when $G$ is considered as a topological group and when it is considered as a topological vector space. It is easy to see that these two notions agree. Now, suppose a locally solid $\ell$-group $G$ has this mild property. So, we prove that in this case, order bounded sets are bounded. But in general, this is not true, consider \cite[Example 4.2]{H}.
\end{rem}

\begin{lem}\label{100}
Suppose $G$ is a locally solid $\ell$-group whose singletons are bounded. Then, every order bounded set in $G$ is bounded.
\end{lem}
\begin{proof}
Suppose $[u,v]$ is an order interval in $G$ and $U$ is an arbitrary neighborhood at the identity in $G$. There is a positive integer $n$ with $(|u|+|v|)\in nU$. So, for each $u\leq w\leq v$, since $|w|\leq |u|+|v|$ and $U$ is solid, we conclude that $w\in nU$, as claimed.
\end{proof}
It is known that every singleton in a topological ring is bounded. So, we have the following observation, too.
\begin{lem}\label{55}
Suppose $X$ is a locally solid $\ell$-ring. Then, every order bounded set in $X$ is bounded.
\end{lem}
\begin{proof}
Suppose $[u,v]$ is an order interval in $X$ and $W$ is an arbitrary zero neighborhood. There is a zero neighborhood $V\subseteq X$  with $V(|u|+|v|)\subseteq W$. So, for each $u\leq x\leq v$, since $|x|\leq |u|+|v|$ and $W$ is solid, we conclude that $Vx\subseteq W$.
\end{proof}
{\bf Observation}. From now on, in this paper, we always assume that all topological groups have this mild property: boundedness of singletons.

Now, we improve \cite[Proposition 2]{Z3}; in fact, the underlying topological group need not be connected, just, it suffices to have boundedness condition for singletons. The proof is essentially the same.
\begin{prop}\label{101}
Suppose $X$ is a topological ring that singletons in the underlying topological group are bounded. Then, we have the following.
\begin{itemize}
		\item[\em (i)] { If $B\subseteq X$ is bounded in the sense of the underlying topological group, then $B$ is bounded}.
		\item[\em (ii)] { If, in addition, $X$ possesses a unity and $B\subseteq X$ is bounded, then, it is bounded in the sense of the topological group}.
		\end{itemize}
\end{prop}

\begin{rem}\label{666}
Note that being unital is a sufficient condition in Proposition \ref{101}; in many classical spaces such as $\ell_p$ for $1\leq p\leq \infty$, $c_0$ and $c_{00}$, it can be verified that notions of boundedness in the sense of topological vector space, underlying topological group, and topological ring ( while they are considered with coordinate-wise multiplication) agree.
\end{rem}
Let us first consider, as an application of $AM$-property, a useful fact about locally solid $\ell$-groups.
\begin{prop}\label{22}
Suppose $G$ is locally solid $\ell$-group. Then, the following are equivalent.
 \begin{itemize}
		\item[\em (i)] { $G$ possesses $AM$ and Levi properties}.
		\item[\em (ii)] { Every order bounded set in $G$ is bounded and vice versa}.
		\end{itemize}
\end{prop}
\begin{proof}
$(i)\to (ii)$. The direct implication is trivial by Lemma \ref{100}. For the converse, assume that $B\subseteq G$ is bounded; W.L.O.G, we may assume that $B$ is solid, otherwise, consider the solid hull of $B$ which is again bounded. So, $B_{+}=\{x\in B, x\geq 0\}$ is also bounded. Assume that $(B_{+})^{\vee}$ is the set of all finite suprema of elements of $B_{+}$. By the $AM$-property, $(B_{+})^{\vee}$ is also bounded. In addition, $(B_{+})^{\vee}$ can be considered as an increasing net in $G_{+}$. So, by the Levi property, $\sup (B_{+})^{\vee}$ exists. But in this case, $\sup B_{+}$ also exists and $\sup (B_{+})^{\vee}=\sup B_{+}$. Put $y=\sup B_{+}$. Therefore, for each $x\in B_{+}$, $x\leq y$; now, it is clear from the relation $B\subseteq B_{+}-B_{+}$ that $B$ is also order bounded.

$(ii)\to (i)$. Suppose $B\subseteq G$ is bounded so that order bounded. Now, it is clear that $B^{\vee}$ is also order bounded and therefore bounded.

Suppose $D$ is an upward directed bounded set in $G_{+}$. So, it is order bounded. Now, $D$ has a supremum since $G$ is order complete.
\end{proof}
Assume that $H=\Bbb Z$ with the discrete topology. It is a locally solid $\ell$-group. The only bounded set is the singleton zero and other singletons are never bounded. So, $H$ possesses the Levi and $AM$ properties. Nevertheless, note that every non-zero singleton is order bounded but not bounded. This justifies importance of the above observation ( boundedness of singletons in a topological group).
Moreover, using Proposition \ref{101}, we obtain the following result for locally solid $\ell$-rings.
\begin{coro}\label{66}
  Suppose $X$ is an order complete locally solid $\ell$-ring with unity. Then, the following are equivalent.
  \begin{itemize}
		\item[\em (i)] { $X$ possesses $AM$ and Levi properties}.
		\item[\em (ii)] { Every order bounded set in $X$ is bounded and vice versa}.
		\end{itemize}
\end{coro}
But the interesting point here is that it is not necessary for locally solid $\ell$-ring $X$ to be unital; more precisely, we improve Corollary \ref{66}. The main idea of the proof is essentially as the same as the proof of Proposition \ref{22}.
\begin{thm}\label{12}
  Suppose $X$ is an order complete locally solid $f$-ring. Then, the following are equivalent.
  \begin{itemize}
		\item[\em (i)] { $X$ possesses $AM$ and Levi properties}.
		\item[\em (ii)] { Every order bounded set in $X$ is bounded and vice versa}.
		\end{itemize}
\end{thm}
\begin{proof}
$(i)\to (ii)$. The direct implication is trivial by Lemma \ref{55}. For the converse, assume that $B\subseteq X$ is bounded; W.L.O.G, we may assume that $B$ is solid, otherwise, consider the solid hull of $B$ which is again bounded by \cite[Lemma 5]{Z2}. So, $B_{+}=\{x\in B, x\geq 0\}$ is also bounded. Assume that $(B_{+})^{\vee}$ is the set of all finite suprema of elements of $B_{+}$. By the $AM$-property, $(B_{+})^{\vee}$ is also bounded. In addition, $(B_{+})^{\vee}$ can be considered as an increasing net in $X_{+}$. So, by the Levi property, $\sup (B_{+})^{\vee}$ exists. But in this case, $\sup B_{+}$ also exists and $\sup (B_{+})^{\vee}=\sup B_{+}$. Put $y=\sup B_{+}$. Therefore, for each $x\in B_{+}$, $x\leq y$; now, it is clear from the relation $B\subseteq B_{+}-B_{+}$ that $B$ is also order bounded.

$(ii)\to (i)$. Suppose $B\subseteq X$ is bounded so that order bounded. Now, it is clear that $B^{\vee}$ is also order bounded and therefore bounded.

Suppose $D$ is an upward directed bounded set in $X_{+}$. So, it is order bounded. Now, $D$ has a supremum since $X$ is order complete.
\end{proof}
Observe that order completeness is essential in the assumptions of Theorem \ref{12} and can not be removed. Consider the ring  $X=C[0,1]$; it possesses $AM$-property. Also, boundedness and order boundedness notions agree in $X$ by \cite[Proposition 2.1]{MZ} and also using this fact that in $C(K)$-spaces, boundedness and order boundedness coincide. But it does not have the Levi property.
\begin{coro}\label{13}
 Suppose $X$ is a  locally solid $\ell$-ring and $Y$ is a locally solid $f$-ring that possesses $AM$ and Levi properties. Then, for a group homomorphism $T:X\to Y$, we have the following observations.
 \begin{itemize}
		\item[\em (i)] { If $T$ is $nr$-bounded, then $T$ is order bounded}.
		\item[\em (ii)] { If $T$ is $br$-bounded, then $T$ is order bounded}.
\item[\em (ii)] { If $T$ is continuous, then $T$ is order bounded}.
		\end{itemize}
\end{coro}
\begin{proof}

$(i)$. Suppose $T$ is $nr$-bounded. So, there is a zero neighborhood $U\subseteq X$ such that $T(U)$ is bounded. Assume that $B\subseteq X$ is bounded in the sense of the underlying topological group. Thus, there exists a positive integer $n$ with $B\subseteq nU$ so that $T(B)\subseteq nT(U)$. This implies that $T(B)$ is bounded.  Now, suppose $A\subseteq X$ is order bounded so that bounded in the sense of the topological group. Using previous argument, we conclude that $T(A)$ is bounded in $X$. Thus, Theorem \ref{12} yields that $T(A)$ is order bounded, as claimed.

$(ii)$. Suppose $A\subseteq X$ is order bounded. Therefore, it is bounded by Lemma \ref{55}. By the assumption, $T(A)$ is also bounded in $Y$. Therefore, Theorem  \ref{12} results in order boundedness of $T(A)$.

$(iii)$. Now, suppose $T$ is continuous. By \cite[Remark 2.4]{KZ}, $T$ is $bb$-bounded in the sense that it maps bounded sets to bounded sets while we consider boundedness in the topological group setting. Now, suppose $A\subseteq X$ is order bounded so that bounded in the sense of the underlying topological group by Lemma \ref{100}. This results in boundedness of $T(A)$ in $Y$ ( again in the topological group sense). By Proposition \ref{101}, we conclude that $T(A)$ is bounded and by Theorem \ref{12}, order bounded, as we wanted.
\end{proof}

By considering Corollary \ref{13} and \cite[Lemma 4, Lemma 5, Lemma 6]{Z2}, we have the following observations.
\begin{coro}
Suppose $X$ is a locally solid $f$-ring  that possesses  $AM$, Fatou, and Levi properties. Then ${\sf Hom_{nr}(X)}$ is a lattice ring.
\end{coro}

\begin{coro}
Suppose $X$ is a locally solid $f$-ring  that possesses  $AM$, Fatou, and Levi properties. Then ${\sf Hom_{br}(X)}$ is a lattice ring.
\end{coro}
\begin{coro}
Suppose $X$ is a locally solid $f$-ring  that possesses  $AM$, Fatou, and Levi properties. Then ${\sf Hom_{cr}(X)}$ is a lattice ring.
\end{coro}

\begin{prop}\label{14}
Suppose $X$ is a locally solid $f$-ring that possesses $AM$ and Levi property and $Y$ is any locally solid $\ell$-ring. Then, every order bounded group homomorphism $T:X\to Y$  is $br$-bounded.
\end{prop}
\begin{proof}
Suppose $B\subseteq X$ is bounded. By Theorem \ref{12}, $B$ is also order bounded. By the assumption, $T(B)$ is order bounded so that bounded by Lemma \ref{55}.
\end{proof}
\begin{rem}
  We can not expect Proposition \ref{14} for either $nr$-bounded group homomorphisms or continuous group homomorphisms. Consider the identity group homomorphism on ${\Bbb R}^{\Bbb N}$. It is order bounded but not an $nr$-bounded group homomorphism by \cite[Example 2.1]{MZ}; observe that ${\Bbb R}^{\Bbb N}$ has $AM$ and Levi properties by Proposition \ref{000} and Proposition \ref{111}.

  Furthermore, suppose $X$ is the additive group $\ell_{\infty}$ with the absolute weak topology, pointwise ordering, and coordinate-wise multiplication and $Y$ is $\ell_{\infty}$ with the coordinate-wise multiplication, pointwise ordering, and the uniform norm topology. Then, the identity group homomorphism $I$ from $X$ to $Y$ is order bounded but not continuous. Observe that $X$ possesses the Levi and $AM$ properties.

\end{rem}

Before, we proceed with another application of $AM$-property, we have the following useful observation. Recall that ${\sf Hom^{b}}(X,Y)$ is the ring of all order bounded group homomorphisms from an $\ell$-ring $X$ into an $\ell$-ring $Y$.
\begin{lem}\label{20}
Suppose $X$ is a locally solid $f$-ring and $Y$ is a locally solid $f$-ring that possesses the Fatou property and is order complete. Then we have the following.
\begin{itemize}
\item[\em (i)] {${\sf Hom^{b}_{nr}}(X,Y)$ is an ideal of ${\sf Hom^{b}}(X,Y)$}.
\item[\em (ii)]{${\sf Hom^{b}_{br}}(X,Y)$ is an ideal of ${\sf Hom^{b}}(X,Y)$}.
\item[\em (iii)] {${\sf Hom^{b}_{cr}}(X,Y)$ is an ideal of ${\sf Hom^{b}}(X,Y)$}.
\end{itemize}
\end{lem}
\begin{proof}
$(i)$. Assume $|T|\leq |S|$ where $T$ is order bounded and $S\in {\sf Hom^{b}_{nr}}(X,Y)$. There exists a zero neighborhood $U\subseteq X$ such that $S(U)$ is bounded. So, for each zero neighborhood $W\subseteq Y$, there is a zero neighborhood $V\subseteq Y$ with $VS(U)\subseteq  W$. Since $U$ is solid, for any $y\in U$, $y^{+}, y^{-}, |y|\in U$. Fix any $x\in U_{+}$. Then $|T|(x)\leq |S|(x)$. In addition, by \cite[Theorem 1]{Z2}, $|S|(x)=\sup\{|S(u)|:|u|\leq x\}$. Since $U$ is solid and $V$ is order closed, we conclude that $V|S|(x)\subseteq W$ so that $V|T|(x)\subseteq W$. Since $|T(x)|\leq |T|(x)$, we see that $V|T(x)|\subseteq W$. So, $VT(x)\subseteq W$. Therefore, $VT(U_{+})\subseteq W$. Since $U\subseteq U_{+}-U_{+}$, we conclude that $T(U)$ is also bounded.

$(ii)$. It is similar to the proof of $(i)$. Just, observe that for a bounded set $B\subseteq X$, W.L.O.G, we may assume that $B$ is solid; otherwise, consider the solid hull of $B$ which is also bounded by \cite[Lemma 5]{Z2}.

$(iii)$. Assume $|T|\leq |S|$ where $T$ is order bounded and $S\in {\sf Hom^{b}_{cr}}(X,Y)$. Choose arbitrary zero neighborhood $W\subseteq Y$. There is a zero neighborhood $V$ with $V-V\subseteq W$. Find any neighborhood $U$ such that $S(U)\subseteq V$. Fix any $x\in U_{+}$. Then, $|T|(x)\leq |S|(x)$. In addition, by \cite[Theorem 1]{Z2}, $|S|(x)=\sup\{|S(u)|:|u|\leq x\}$. Since $U$ is solid and also $V$ and $W$ are order closed, we conclude that $|S|(x)\in V$ so that $|T|(x)\in V$. Since $|T(x)|\leq |T|(x)$, we see that $|T(x)|\in V$. So, $T(x)\in V$. Therefore, $T(U_{+})\subseteq  V$. Since $U\subseteq U_{+}-U_{+}$, we conclude that $T(U)\subseteq T(U_{+})-T(U_{+})\subseteq V-V\subseteq W$, as desired.
\end{proof}
As a consequence, we state a domination property for each class of bounded order bounded group homomorphisms.
\begin{coro}
Suppose $X$ is a locally solid $f$-ring and $Y$ is a locally solid $f$-ring that possesses the Fatou property and is order complete. Moreover, assume that $T,S:X\to Y$ are group homomorphisms such that $0\leq T\leq S$. Then we have the following.
\begin{itemize}
\item[\em (i)] {If $S\in {\sf Hom^{b}_{nr}}(X,Y)$ then  $T\in {\sf Hom^{b}}(X,Y)$}.
\item[\em (ii)]{If $S\in {\sf Hom^{b}_{br}}(X,Y)$ then  $T\in {\sf Hom^{b}}(X,Y)$}.
\item[\em (iii)] {If $S\in {\sf Hom^{b}_{cr}}(X,Y)$ then  $T\in {\sf Hom^{b}}(X,Y)$}.
\end{itemize}
\end{coro}



\begin{thm}\label{30}
Suppose $X$ is an order complete locally solid $f$-ring with unity and the Fatou property. Then ${\sf Hom^{b}_{br}(X)}$ has the Levi property if and only if so is $X$.
\end{thm}
\begin{proof}
Suppose $(T_{\alpha})$ is a bounded increasing net in ${\sf Hom^{b}_{br}(X)}_{+}$. Therefore, for every  bounded set $B\subseteq X$, it follows that $(T_{\alpha}(B))$ is uniformly bounded for each $\alpha$. Thus, for each $x\in X_{+}$, the net $(T_{\alpha}(x))$ is bounded and increasing in $X$ so that it has a supremum, namely, $\alpha_x$. Define $T_{\alpha}:X_{+}\to X_{+}$ via $T_{\alpha}(x)=\alpha_x$. It is an additive map; it is easy to see that $\alpha_{x+y}\leq \alpha_x+\alpha_y$. For the converse, fix any $\alpha_0$. For each $\alpha\geq\alpha_0$, we have $T_{\alpha}(x)\leq \alpha_{x+y}-T_{\alpha}(y)\leq \alpha_{x+y}-T_{\alpha_0}(y)$ so that $\alpha_x\leq \alpha_{x+y}-T_{\alpha_0}(y)$. Since $\alpha_0$ was arbitrary, we conclude that $\alpha_{x}+\alpha_{y}\leq\alpha_{x+y}$. By \cite[Lemma 1]{Z2}, it extends to a positive group homomorphism $T:X\to X$. We need to show that $T\in {\sf Hom^{b}_{br}(X)}$. It is clear that $T$ is order bounded. Suppose $W$ is an arbitrary zero neighborhood in $X$. There is a zero neighborhood $V$ with $VT_{\alpha}(B)\subseteq W$. This means that $VT(B)\subseteq W$ since $W$ has the Fatou property and also using \cite[Theorem 3.15]{J}.

For the converse, assume that $(x_{\alpha})$ is a bounded increasing net in $X_{+}$. Define $T_{\alpha}:X\to X$ with $T_{\alpha}(x)=xx_{\alpha}$. It is easy to see that each $T_{\alpha}$ is $br$-bounded as well as order bounded. Fix a bounded set $B\subseteq X$. Suppose $W\subseteq X$ is an arbitrary zero neighborhood. Since the net $(x_{\alpha})$ is bounded, there exists a zero neighborhood $V\subseteq X$ such that $V(Bx_{\alpha})\subseteq W$ for each $\alpha$. It follows that $(T_{\alpha}(B))$ is bounded and increasing in ${\sf Hom^{b}_{br}(X)}$. Thus, by the assumption, $T_{\alpha}\uparrow T$ for some $T\in {\sf Hom^{b}_{br}(X)}$. Therefore, $T_{\alpha}({\sf 1})\uparrow T({\sf 1})$; that is $x_{\alpha}\uparrow T({\sf 1})$, as claimed.
\end{proof}

\begin{lem}\label{444}
Suppose $X$ is a locally bounded order complete locally solid $f$-ring with unity and the Fatou property. Then ${\sf Hom^{b}_{nr}(X)}={\sf Hom^{b}_{br}(X)}$.

\end{lem}
\begin{proof}
Assume that $X$ is locally bounded and a group homomorphisms $T$ on $X$ is $nr$-bounded. So, there exists a zero neighborhood $U\subseteq X$ such that $T(U)$ is bounded in $X$. Suppose $B\subseteq X$ is bounded. By, Proposition \ref{101}, it is bounded also in the sense of the underlying topological group. Find positive integer $n$ with $B\subseteq nU$ so that $T(B)\subseteq n T(U)$. This means that $T$ is $br$-bounded. Furthermore, by the assumption, every $br$-bounded group homomorphism is also $nr$-bounded, as claimed.
\end{proof}
Compatible with Lemma \ref{444} and Theorem \ref{30}, we have the following.
\begin{coro}
Suppose $X$ is a locally bounded order complete locally solid $f$-ring with  unity and the Fatou property. Then ${\sf Hom^{b}_{nr}(X)}$
 has the Levi property if and only if so is $X$.
\end{coro}
\begin{thm}
Suppose $X$ is an order complete locally solid $f$-ring with  unity and the Fatou property. Then ${\sf Hom^{b}_{cr}(X)}$ has the Levi property if and only if so is $X$.
\end{thm}
\begin{proof}
Suppose $(T_{\alpha})$ is a bounded increasing net in ${\sf Hom^{b}_{cr}(X)}_{+}$. This implies that the set $(T_{\alpha})$ is equicontinuous in the sense that for each zero neighborhood $W\subseteq X$, there is a zero neighborhood $U$ such that $T_{\alpha}(U)\subseteq W$ for each $\alpha$. So, for each $x\in X_{+}$, the net $(T_{\alpha}(x))$ is bounded and increasing in $X$ so that has a supremum, namely, $\alpha_x$. Define $T_{\alpha}:X_{+}\to X_{+}$ via $T_{\alpha}(x)=\alpha_x$. It is an additive map. By \cite[Lemma 1]{Z2}, it extends to a positive group homomorphism $T:X\to X$. We need to show that $T\in {\sf Hom^{b}_{cr}(X)}$. It is clear that $T$ is order bounded. Moreover, it can be easily seen that $T(U) \subseteq W$ since $W$ has the Fatou property.

For the converse, assume that $(x_{\alpha})$ is a bounded increasing net in $X_{+}$. Define $T_{\alpha}:X\to X$ via $T_{\alpha}(x)=xx_{\alpha}$. It is easy to see that each $T_{\alpha}$ is continuous as well as order bounded. For an arbitrary zero neighborhood $W\subseteq X$, there is a zero neighborhood $U$ such that $U(x_{\alpha})\subseteq W$. It follows that $(T_{\alpha})$ is bounded and increasing. Thus, by the assumption, $T_{\alpha}\uparrow T$ for some $T\in {\sf Hom^{b}_{cr}(X)}$. Therefore, $T_{\alpha}({\sf1})\uparrow T({\sf1})$; that is $x_{\alpha}\uparrow T({\sf1})$, as claimed.
\end{proof}

in this step, we recall a ring version of \cite[Theorem 1.35]{AB}. The proof is essentially the same.
\begin{lem}\label{333}
Suppose $X$ is an $\ell$-ring and $I$ is an ideal of $X$. Then for a set $D\subseteq I_{+}$, $D\downarrow 0$ in $X$ if and only if $D\downarrow 0$ in $I$.
\end{lem}

\begin{prop}\label{555}
Suppose $X$ is an order complete locally solid $f$-ring with unity and the Fatou property. If ${\sf Hom^{b}_{br}(X)}$ has the Lebesgue property then so is $X$.
\end{prop}
\begin{proof}
Suppose $(x_{\alpha})$ is a net in $X$ such that $x_{\alpha}\downarrow 0$. Define $T_{\alpha}:X\to X$ with $T_{\alpha}(x)=xx_{\alpha}$. It is easy to see that each $T_{\alpha}$ is $br$-bounded as well as order bounded. First, note that by using \cite[Theorem 1]{Z2}, we conclude that $T_{\alpha}\downarrow 0$ in ${\sf Hom^{b}(X)}$ if and only if $T_{\alpha}(x)\downarrow 0$ for each $x\in X_{+}$.  Furthermore, observe that by Lemma \ref{20} and Lemma \ref{333}, we conclude that $T_{\alpha}\downarrow 0$ in  ${\sf Hom^{b}_{br}(X)}$.
So, by the assumption, $T_{\alpha}\rightarrow 0$ uniformly on bounded sets. Therefore, $T_{\alpha}({\sf 1})\rightarrow 0$ in $X$; this means $(x_{\alpha})$ is a null net in $X$, as claimed.

\end{proof}
By using Lemma \ref{444} and Proposition \ref{555}, one may consider the following.
\begin{coro}
Suppose $X$ is a locally bounded order complete locally solid $f$-ring with  unity and the Fatou property. If ${\sf Hom^{b}_{nr}(X)}$ has the Lebesgue property then so is $X$.
\end{coro}

For the converse of Proposition \ref{555}, we have the following.
\begin{thm}\label{1}
Suppose $X$ is a locally solid $f$-ring that possesses $AM$ and Levi properties and $Y$ is an order complete locally solid $f$-ring. If $Y$ has the Lebesgue property, then so is ${\sf Hom^{b}}(X,Y)$.
\end{thm}
\begin{proof}
First, observe that by Proposition \ref{14}, ${\sf Hom^{b}}(X,Y)={\sf Hom ^{b}_{br}}(X,Y)$. Suppose $(T_{\alpha})_{\alpha\in I}$ is a net in ${\sf Hom ^{b}_{br}}(X,Y)$ such that $T_{\alpha}\downarrow 0$.
Choose a bounded set $B\subseteq X$; W.L.O.G, we may assume that $B$ is solid, otherwise, consider the solid hull of $B$ which is certainly bounded by \cite[Lemma 5]{Z2}. By Corollary \ref{13}, $B$ is order bounded.
The remaining part of the  proof has the same line as in \cite[Theorem 5]{Z4}.
Put $A=\{T_{\alpha}(x), \alpha\in I, x\in B_{+}\}$. Again, W.L.O.G, assume that $B_{+}=[0,u]$, in which $u\in X_{+}$. Define $\Lambda=I\times [0,u]$. Certainly, $\Lambda$ is a directed set while we consider it with the lexicographic order, namely, $(\alpha,x)\leq (\beta,y)$ if $\alpha<\beta$ or $\alpha=\beta$ and $x\leq y$. In notation, $A=(y_{\lambda})_{\lambda\in \Lambda}\geq {\sf 0}$. So, by considering $A^{\wedge}$, one can assume $A$ as a decreasing  net in $Y_{+}$. Therefore, it has an infimum. We claim that $A\downarrow 0$; otherwise, there is a $0\neq y\in Y_{+}$ such that $y_{\lambda}\geq y$ for each $\lambda\in \Lambda$. Therefore, for each $\alpha$ and each $x\in B_{+}$, $T_{\alpha}(x)\geq y$ which is in contradiction with $T_{\alpha}\downarrow 0$. By the assumption, $y_{\lambda}\rightarrow 0$ in $Y$. Therefore, for an arbitrary zero neighborhood $V\subseteq Y$, there exists a $\lambda_0=(\alpha_0,x_0)$ such that $y_{\lambda}\in V$ for each $\lambda\geq\lambda_0$. Suppose $\lambda=(\alpha,x)$. So, for each $\alpha>\alpha_0$ and for each $x\in B_{+}$,  $T_{\alpha}(x)\in V$. Since $B\subseteq B_{+}-B_{+}$, we conclude that  $T_{\alpha}\rightarrow 0$ in ${\sf Hom ^{b}_{br}}(X,Y)$.
\end{proof}
\begin{rem}
Observe that hypotheses in Theorem \ref{1} are essential and can not be removed. Consider locally solid $\ell$-ring $X=c_0$ with norm topology, pointwise ordering and coordinate-wise multiplication.
It possesses the $AM$-property and its topology is Lebesgue but it fails to have the Levi property. Suppose $(P_n)$ is the sequence of coordinate-wise group homomorphisms on $X$, namely $P_n((x_m))=(x_1,\ldots,x_n,0,\ldots)$. Each $P_n$ is $br$-bounded and $P_n\uparrow I$, where $I$ is the identity group homomorphism on $X$. But $P_n\nrightarrow I$ uniformly on the unit ball of $X$.

Moreover, consider $Y=\ell_1$ with norm topology, pointwise ordering and coordinate-wise multiplication; it has the Lebesgue and the Levi properties but it fails to have the $AM$-property. Again, if $(P_n)$ is the sequence of coordinate-wise group homomorphisms on $Y$, $P_n\uparrow I$ but certainly not in the topology  of uniform convergence on bounded sets.

Just observe that  by Remark \ref{666}, the notions of boundedness in topological vector space and topological ring setting coincide.
\end{rem}
\begin{prop}
Suppose $X$ is an order complete locally solid $f$-ring with  unity and the Fatou property. If ${\sf Hom^{b}_{cr}(X)}$ has the Lebesgue property then so is $X$.
\end{prop}
\begin{proof}
Suppose $(x_{\alpha})$ is a net in $X$ such that $x_{\alpha}\downarrow 0$. Define $T_{\alpha}:X\to X$ with $T_{\alpha}(x)=xx_{\alpha}$. It is easy to see that each $T_{\alpha}$ is continuous as well as order bounded. Observe that by Lemma \ref{20} and Lemma \ref{333}, we conclude that $T_{\alpha}\downarrow 0$ in  ${\sf Hom^{b}_{cr}(X)}$.
So, by the assumption, $T_{\alpha}\rightarrow 0$ in the $cr$-convergence topology. Therefore, $T_{\alpha}({\sf 1})\rightarrow 0$ in $X$; this means $(x_{\alpha})$ is a null net in $X$, as claimed.


\end{proof}

\end{document}